\documentclass{amsart}

\usepackage{graphicx}

\newtheorem{theorem}{Theorem}
\newtheorem*{theorem*}{Theorem}
\newtheorem{lemma}[theorem]{Lemma}
\newtheorem{proposition}[theorem]{Proposition}
\newtheorem{corollary}[theorem]{Corollary}
\newtheorem{remark}[theorem]{Remark}
\newtheorem*{question*}{Question}
\newtheorem*{conjecture*}{Conjecture}
\DeclareMathOperator{\D}{D}
\DeclareMathOperator{\conv}{conv}
\DeclareMathOperator{\vol}{vol}
\DeclareMathOperator{\card}{card}
\DeclareMathOperator{\aff}{aff}

\begin{document}

\title[A Geometric Solution to a Maximin Problem Involving Determinants]{A Geometric Solution to a Maximin Problem Involving Determinants of Sets of Unit Vectors in Finite Dimensional Real or Complex Vector Spaces}

\author{Mark Fincher}
\address{Department of Mathematics, University of Pittsburgh, 
Pittsburgh, PA 15213}
\email{MEF98@pitt.edu}

\begin{abstract}
Given $n+1$ unit vectors in $\mathbf{R}^n$ or $\mathbf{C}^n,$ consider the absolute values of the determinants of the vectors taken $n$ at a time. By taking a geometric perspective, we show that the minimum of these determinants is maximized when the vectors point from the origin to the vertices of a regular simplex inscribed in the unit sphere in $\mathbf{R}^n,$ even in the complex case. We also discuss variations on this problem and a few connections to other problems.
\end{abstract}

\maketitle

\section{Introduction} 

Let $\mathcal{V}$ be an $n$-dimensional real or complex vector space. Let $V = \{w_0,\dots,w_k\}$ be a set of $k+1 > n$ distinct unit vectors in $\mathcal{V}.$ Let
$$\D(V) = \min\{|\det(w_{j_1},w_{j_2},\dots,w_{j_n})| : 0 \leq j_1 < j_2 < \dots < j_n \leq k\}.$$
For some given $k,$ consider the problem of finding optimal sets $V$ so that $\D(V)$ is as large as possible. What is the largest $\D(V),$ and how can one describe the optimal sets? In general, we call this the \textit{maximin determinants problem for $k+1$ unit vectors in $\mathcal{V}$}. 

In this paper, we solve the problem for when $n$ is arbitrary, $k=n,$ and $\mathcal{V}$ is $\mathbf{R}^n$ or $\mathbf{C}^n.$ The solution for the complex case involves considering the underlying real vector space of double the dimension and looking at the polytopes with $2n+2$ vertices with maximal volume over all such polytopes inscribed in the unit sphere in $2n$-dimensional Euclidean spaces. The description of these polytopes was recently given by a result of Horv\'ath and L\'angi in \cite{HorvathLangi}. The result for the real case will be a simple corollary.

Working up to this, in section \ref{motivating example} we look informally at $3$ unit vectors in $\mathbf{R}^2.$ There, we motivate a kind of geometric argument similar to the argument later used for $n+1$ unit vectors in $\mathbf{C}^n.$ It is a good idea to have this simple case in mind before going into the general complex case.

In section \ref{preliminaries}, the pertinent definitions and facts to be used concerning Euclidean geometry of arbitrary dimension are provided. In particular, we define simplices, explain some of their properties, and give the aforementioned result of Horv\'ath and L\'angi.

Next, in section \ref{main result section} we prove the main result concerning $n+1$ unit vectors in $\mathbf{C}^n.$ The corollary for $\mathbf{R}^n$ is then given.

Finally, in the last section we provide some motivation for studying the maximin determinants problem, give the solution for any $k+1$ unit vectors in $\mathbf{R}^2,$ and discuss further cases.

\section{A Look at $3$ Vectors in $\mathbf{R}^2$} \label{motivating example}

The maximin determinants problem for $k+1$ unit vectors in $\mathbf{R}^2$ is easy for arbitrary $k+1>2.$ The general solution is given in section \ref{extras}. However, there is some value in looking at the specific case of $3$ vectors in $\mathbf{R}^2,$ because we can develop a useful idea. This section is for motivating purposes and is not intended to be formal.

Let $V = \{w_0,w_1,w_2\}$ be a set of $3$ unit vectors in $\mathbf{R}^2$ which form a triangle containing the origin when lines are drawn connecting the tips of the vectors. This triangle is inscribed in the unit circle, and might look as in figure \ref{figure1}.

\begin{figure}
	\centering
	\includegraphics[width=0.6\textwidth]{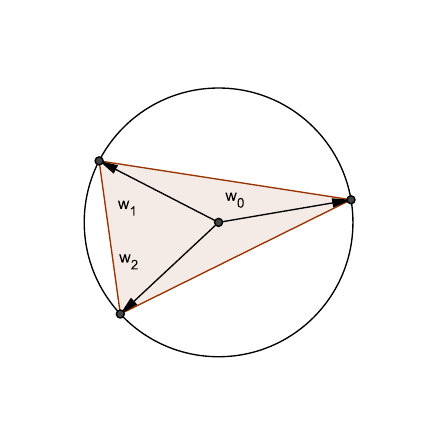}
	\caption{}\label{figure1}	
\end{figure}

The triangle with vertices at $w_0,w_1,$ and $w_2$ is partitioned into three smaller triangles, each with one vertex at the origin. Hence, we can write the area of the larger triangle in terms of the sum of the areas of the smaller ones, which can be written in terms of determinants.
$$\text{Area}_V = \frac{1}{2}(|\det(w_0,w_1)|+|\det(w_0,w_2)|+|\det(w_1,w_2)|).$$
Let $R = \{v_0,v_1,v_2\}$ be a set of unit vectors with tips at the vertices of an equilateral triangle inscribed in the unit circle. Such a triangle contains the origin and hence we can write its area as
$$\text{Area}_R = \frac{1}{2}(|\det(v_0,v_1)|+|\det(v_0,v_2)|+|\det(v_1,v_2)|).$$
Since the vectors in $R$ form an equilateral triangle, we have
$$|\det(v_0,v_1)|=|\det(v_0,v_2)|=|\det(v_1,v_2)|=\D(R).$$
Furthermore, it is a simple fact that the equilateral triangle has the greatest area of all triangles which can be inscribed in a circle, giving us
$$\text{Area}_V \leq \text{Area}_R,$$
and hence
$$\frac{1}{2}(|\det(w_0,w_1)|+|\det(w_0,w_2)|+|\det(w_1,w_2)|) \leq \frac{3}{2}\D(R).$$
Of course, this means
$$\D(V) \leq \D(R).$$
It seems to then be the case that an optimal configuration for maximizing the minimum determinant could be when the three vectors point to the vertices of an equilateral triangle. All that is left to check is what $\D(V)$ could be if $w_0,w_1,$ and $w_2$ determine a triangle which does not contain the origin, like in figure \ref{figure2}.

\begin{figure}
	\centering
	\includegraphics[width=0.6\textwidth]{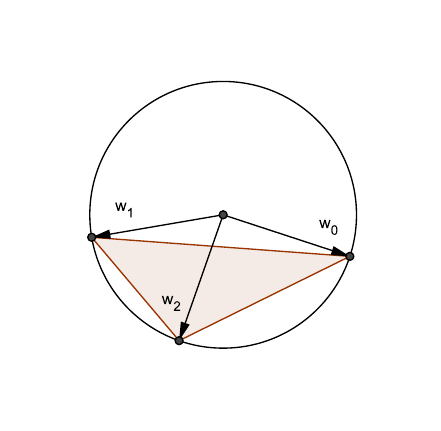}
	\caption{}\label{figure2}	
\end{figure}

In this case, we cannot write the area of the triangle in terms of the sum of the absolute values of the determinants. However, we can try to cook up another set of $3$ unit vectors, call it $\tilde{V},$ such that the triangle formed by the vectors of $\tilde{V}$ contains the origin and $\D(V)=\D(\tilde{V}).$ In this situation, let $\tilde{V}=\{w_0,w_1,-w_2\}.$ The triangle formed by these vectors does contain the origin, so by the previous reasoning
$$\D(\tilde{V}) \leq \D(R).$$
Clearly $\D(V)=\D(\tilde{V}),$ so $\D(V) \leq \D(R).$

In conclusion, it seems like an optimal configuration is when the three vectors point to the vertices of an equilateral triangle. Of course this is not the only optimal configuration, since any transformation of these vectors which preserves the absolute values of the determinants will also be optimal (for instance, multiplying some of the vectors by $-1$).

Let us summarize the few critical facts which made this argument work. Firstly, when the triangle with vertices determined by $V$ contains the origin, we can express its area in terms of the sum of the absolute values of the determinants of $V.$ Secondly, there exists a set of unit vectors $R$ determining a triangle with largest area such that the origin is inside this triangle and the determinants of the vectors of $R$ are equal in magnitude. Thirdly, if the triangle determined by $V$ does not contain the origin, then we can find another set $\tilde{V}$ so that its triangle does contain the origin and $\D(V)=\D(\tilde{V}).$

We will rely on generalized versions of these facts later.

\section{Preliminaries} \label{preliminaries}

\subsection{Simplices}

For $V = \{w_0,\dots,w_k\} \subset \mathbf{R}^n,$ with possibly $k<n,$ define the \textit{convex hull} of $V$ as
$$\conv(V) = \{c_0w_0 + \dots + c_kw_k : \text{ each } c_j \geq 0 \text{ and } \sum_{j=0}^{k} c_j =1\}.$$
We view $\conv(V)$ as a polytope with vertices at the tips of some or all of the vectors of $V.$ For instance, if the $w_j$ are unit vectors, then $\conv(V)$ is a polytope with vertices at the tips of each of the $w_j.$ For convenience, we sometimes refer to the vectors $w_j$ as the vertices of $\conv(V).$ We let \textit{face} and \textit{edge} have their normal meaning for polytopes. The unit sphere in $\mathbf{R}^n$ with center at $0$ we denote $S^{n-1}$.

We say $V$ is an \textit{affinely independent} set if $$\{w_1 - w_0,\dots,w_k - w_0\}$$ is a linearly independent set, and we call $V$ an \textit{affinely dependent} set otherwise. If $V$ is affinely independent, then $\conv(V)$ is a \textit{k-simplex} with vertices at each $w_j$, or just a simplex.

Let $w_j \cdot w_{\ell}$ mean the standard dot product of $w_j$ and $w_{\ell}.$ The \textit{edge length} between $w_j$ and $w_\ell$ is defined to be $\sqrt{(w_j - w_\ell)\cdot (w_j - w_\ell)}.$ For $\nu \in \mathbf{R}^n,$ not equal to the zero vector, and $c \in \mathbf{R},$ we call $\{x \in \mathbf{R}^n : x\cdot \nu = c\}$ an $(n-1)$-dimensional \textit{hyperplane}.

We let $\vol(\conv(V))$ mean the volume of $\conv(V),$ where the dimension of the volume is stated or should be clear from the context. We will make use of the following well-known formula for the volume of a simplex.

\begin{proposition}[{\cite[\S8.4]{Sommerville}}] \label{simplex volume formula}
	Let $V = \{w_0, \ldots, w_k\} \subset \mathbf{R}^n$ such that $\conv(V)$ is a $k$-simplex. The $k$-dimensional volume of $\conv(V)$ is given by
	$$\vol(\conv(V)) = \frac{1}{k!}|\det(w_1-w_0,\dots,w_k-w_0)|.$$
\end{proposition}

A \textit{regular simplex} is a simplex with all edge lengths equal.

\begin{theorem}[{\cite[Cor.~1]{HorvathLangi}}] \label{regmaxvolume}
	If $P$ is an $n$-simplex inscribed in $S^{n-1}$ with maximal volume over all $n$-simplices inscribed in $S^{n-1}$, then $P$ is a regular $n$-simplex. 
\end{theorem}

\begin{proposition} \label{regdetsformula}
	When $v_0,\dots,v_n \in \mathbf{R}^n$ are unit vectors and lie at the vertices of a regular simplex, then $$|\det(v_{j_1},\dots,v_{j_n})|=\sqrt{\frac{(n+1)^{n-1}}{n^n}}$$ for all $0\leq j_1<\dots <j_n\leq n$.
\end{proposition}

\begin{proof}
	Let $0\leq j_1<\dots <j_n\leq n$. Since $v_0,\dots,v_n$ are unit vectors and the vertices of a regular simplex, it is a simple fact that
	
	$$
	v_j\cdot v_k=\left\{\begin{array}{rl}\displaystyle-\frac{1}{n}&
	\qquad\textnormal{if~}j\ne k\\
	\noalign{\vspace{6pt}}
	1&\qquad\textnormal{if~}j=k.
	\end{array}\right.
	$$
	Therefore,
	$$
	[\det(v_{j_1},\ldots,v_{j_n})]^2 
	= \det(A),$$
	where
	$$ A = \begin{bmatrix}
	1 & -1/n & \ldots & -1/n \\
	-1/n & 1 & \ddots & -1/n \\
	\vdots & \ddots & \ddots & \vdots\\
	-1/n & -1/n & \ldots & 1
	\end{bmatrix}.
	$$
	To compute the determinant of $A$, we will figure out its eigenvalues and then take their product. We can see one eigenvalue of $A$ is $1/n$ because 
	$$A\begin{bmatrix}1 & 1 & \ldots & 1\end{bmatrix}^{T}=\begin{bmatrix}1/n & 1/n & \ldots & 1/n\end{bmatrix}^{T}.$$ 
	Now see that $(n+1)/n$ is also an eigenvalue, because
	$$A\begin{bmatrix}1 & -1 & 0 & \ldots & 0\end{bmatrix}^{T}=\begin{bmatrix}(n+1)/n & -(n+1)/n & 0 & \ldots & 0\end{bmatrix}^{T}.$$ 
	In fact, all the remaining eigenvalues must be $(n+1)/n$ also because in general 
	$$A\begin{bmatrix}1 & 0 & \ldots & -1 & \ldots & 0\end{bmatrix}^{T}=\begin{bmatrix}(n+1)/n & 0 & \ldots & -(n+1)/n & \ldots & 0\end{bmatrix}^{T}$$ 
	Therefore, $n-1$ of the eigenvalues of $A$ are $\frac{n+1}{n}$ and the final eigenvalue of $A$ is $1/n$. Because the determinant of a square matrix is the product of its eigenvalues, we have
	$$\det(A)=\frac{(n+1)^{n-1}}{n^n},$$
	and therefore,
	$$|\det(v_{j_1},\dots,v_{j_n})| = \sqrt{\frac{(n+1)^{n-1}}{n^n}}.$$
\end{proof}

\subsection{Maximal Volume Polytopes with $n+2$ Vertices Inscribed in the Unit Sphere in $\mathbf{R}^n$}

The following theorem is due to Horv\'ath and L\'angi.

\begin{theorem}[{\cite[Th.~2]{HorvathLangi}}] \label{nplus2max}
	Let $V$ be a set of $n+2$ unit vectors in $\mathbf{R}^n$ such that $\vol(\conv(V))$ is maximal over all sets of $n+2$ unit vectors in $\mathbf{R}^n.$ Then there exist disjoint $V_1$ and $V_2$ with
	\begin{itemize}
		\item $V_1 \cup V_2 = V$.
		\item $\card(V_1)=\lfloor{n/2}\rfloor + 1$.
		\item $\card(V_2)=\lceil{n/2}\rceil + 1$.
		\item $V_1$ and $V_2$ are contained in orthogonal linear subspaces of $\mathbf{R}^n.$
		\item $\conv(V_1)$ and $\conv(V_2)$ are regular simplices.
	\end{itemize}
	In this case, we have
	$$\vol(\conv(V)) = \frac{1}{n!}\frac{(\lfloor{n/2}\rfloor +1)^{\frac{\lfloor{n/2}\rfloor +1}{2}}(\lceil{n/2}\rceil +1)^{\frac{\lceil{n/2}\rceil +1}{2}} }{{\lfloor{n/2}\rfloor}^{\frac{\lfloor{n/2}\rfloor}{2}}{\lceil{n/2}\rceil}^{\frac{\lceil{n/2}\rceil}{2}}}.$$ 
\end{theorem}

An idea for a proof is to use a theorem of Radon (see \cite[Th.~1.2]{Eckhoff}) to determine the appropriate partition $V_1$ and $V_2$ so that
$$\vol(\conv(V)) = \binom{n}{m-1}^{-1}\vol(\conv(V_1))\vol(\conv(V'_2)),$$
where $m = \card(V_1)$ and $V'_2$ is the orthogonal projection of $V_2$ on a subspace orthogonal to $\aff(V_1).$ While mentioning this, Horv\'ath and L\'angi use the Gale Transform instead of Radon's theorem to find the $V_1$ and $V_2.$

Kind and Kleinschmidt \cite{KiKl} used Radon's theorem in this way to solve the problem of maximizing $\vol(\conv(V))$ when the diameter of $V$ is fixed. This isodiametric volume problem and the problem of finding maximal volume polytopes inscribed in the unit sphere are among a host of related problems in $n$-dimensional geometry. For more discussion, see the introduction to \cite{HorvathLangi}.

\section{The Maximin Determinants Problem for $n+1$ Unit Vectors in $\mathbf{C}^n$} \label{main result section}

Let $V = \{w_0,\dots,w_n\} \subset \mathbf{C}^n,$ where  $w_j=[z_{j,1}, \dots ,z_{j,n}]^T$ with $z_{j,k} = x_{j,k} + iy_{j,k}$. We assume the $w_j$ are distinct, so $V$ actually has $n+1$ elements. Call
$$u_j=[x_{j,1}, \ldots, x_{j,n}, y_{j,1}, \dots y_{j,n}]^T  \text{ and } v_j=[-y_{j,1}, \ldots -y_{j,n}, x_{j,1}, \dots, x_{j,n}]^T$$
the \textit{associated real vectors} of $w_j.$ Note for all $j$ we have $u_j,v_j \in \mathbf{R}^{2n}$ and $u_j \cdot v_j = 0.$

Denote the set of all associated real vectors of the vectors of $V$ as $\ddot{V}.$ If the $w_j$ are unit vectors then $\conv(\ddot{V})$ is a convex polytope inscribed in $S^{2n-1}.$

If $y_{j,k} = 0$ for every $j$ and $k$ and $\{[x_{j,1},\dots,x_{j,n}]^T:0 \leq j \leq n\}$ is the vertex set of a regular simplex in $\mathbf{R}^n,$ then say that $V$ is the vertex set of a \textit{real regular simplex} in $\mathbf{C}^n.$

\begin{theorem} \label{maxvolumecomplex}
	Let $V = \{w_0,\dots,w_n\} \subset \mathbf{C}^n$ be a set of unit vectors and the vertex set of a real regular simplex. Then $\vol(\conv(\ddot{V}))$ is maximal over all convex polytopes with $2n + 2$ vertices inscribed in $S^{2n-1}.$
\end{theorem}

\begin{proof}
	Let $V = \{w_0,\dots,w_n\} \subset \mathbf{C}^n$ be unit vectors and the vertex set of a real regular simplex. Let $V_1 = \{u_j:0 \leq j \leq n\}$ and $V_2 = \{v_j: 0 \leq j \leq n\}. $ For all $j,$ $k,$ we have $u_j \cdot v_k = 0$. Therefore, $\conv(V_1)$ and $\conv(V_2)$ are contained in orthogonal linear subspaces, and each have dimension $n.$ Furthermore, $\conv(V_1)$ and $\conv(V_2)$ are regular simplices in their respective subspaces. So, by Theorem \ref{nplus2max}, $\vol(\conv(\ddot{V}))$ is maximal over all convex polytopes with $2n+2$ vertices inscribed in $S^{2n-1}.$
\end{proof}

\begin{proposition}[\cite{StackExchange}] \label{stackexchange}
	Let $Z$ be an $n$-by-$n$ complex matrix with real and imaginary parts given by $Z=X+iY$. Then
	$$|\det Z|^2=\det\begin{bmatrix} X & -Y\\ Y & X \end{bmatrix}.$$
\end{proposition}
\begin{proof}
	Let $A=\begin{bmatrix} X & iY \\ iY & X \end{bmatrix}$. Then
	\begin{align*}
	\det(A)&=\det\begin{bmatrix} X + iY & iY \\ X + iY & X
	\end{bmatrix} \\
	&= \det\begin{bmatrix} I & iY \\ I & X \end{bmatrix}\det\begin{bmatrix} X+iY & 0 \\ 0 & I \end{bmatrix} \\
	&=\det(X-iY)\det(X+iY)
	=|\det{Z}|^2.
	\end{align*}
	Now we note that
	$$\begin{bmatrix} I & 0 \\ 0 & -iI \end{bmatrix}\begin{bmatrix} X & iY \\ iY & X \end{bmatrix}
	\begin{bmatrix} I & 0 \\ 0 & iI \end{bmatrix} = \begin{bmatrix} X & -Y \\ Y & X \end{bmatrix},$$
	where
	$\begin{bmatrix} I & 0 \\ 0 & -iI \end{bmatrix}\begin{bmatrix} I & 0 \\ 0 & iI \end{bmatrix} = I.$
	Therefore,
	$$\det\begin{bmatrix} X & iY \\ iY & X \end{bmatrix} = \det\begin{bmatrix} X & -Y \\ Y & X \end{bmatrix} = |\det{Z}|^2.$$
\end{proof}

Note that, as a result, $|\det(w_{j_1},\dots, w_{j_n})|^2=\det(u_{j_1},\dots,u_{j_n}, v_{j_n},\dots,v_{j_n})$ for all $0 \leq j_1 < j_2 < \dots < j_n \leq n.$

\begin{lemma} \label{origininsidecomplex}
	Let $w_0,\ldots,w_n \in \mathbf{C}^n$. There exist $\tilde{w}_0,\ldots,\tilde{w}_n \in{\mathbf{C}^n}$ such that the following hold.
	\begin{enumerate}
		\item $|w_j|=|\tilde{w}_j|$ for each $j.$
		\item There exist nonnegative real numbers $r_0,\dots,r_n$ such that $r_0 + \dots + r_n = 1$ and $r_0\tilde{w}_0 + \dots + r_n\tilde{w}_n=0$.
		\item $|\det(w_{j_1}, \ldots, w_{j_n})| = |\det(\tilde{w}_{j_1}, \ldots, \tilde{w}_{j_n})|$ for all $0 \leq j_1 < j_2 < \dots < j_n \leq n.$
	\end{enumerate}
\end{lemma}
\begin{proof}
	Let $w_0,\ldots,w_n \in \mathbf{C}^n$. Because they must be linearly dependent, there exist $\gamma_0,\dots,\gamma_n \in \mathbf{C}$, not all zero, such that $\gamma_0 z_0 + \ldots + \gamma_n z_n = 0.$ Let 
	$$r_j = \frac{|\gamma_j|}{|\gamma_0|+\ldots+|\gamma_n|}.$$ 
	Note that $r_j \geq 0$ for each $j$ and $\sum r_j = 1$. Now, choose $\theta_j$ so that
	$$r_je^{i\theta_j}=\frac{\gamma_j}{|\gamma_0|+\ldots+|\gamma_n|}.$$ 
	We then have $r_0(e^{i\theta_0}z_0) + \ldots + r_n(e^{i\theta_n}z_n) = 0$. Let $\tilde{w}_j = e^{i\theta_j}w_j$. Then $(1)$ and $(2)$ are clear, and to see $(3)$ observe
	$$|\det(\tilde{w}_{j_1}, \ldots, \tilde{w}_{j_n})| = |\det(e^{i\theta_{j_1}}w_{j_1}, \ldots, e^{i\theta_{j_n}}w_{j_n})| = |\det(w_{j_1}, \ldots, w_{j_n})|. \qedhere$$ 
\end{proof}

For a set $V = \{w_0,\dots,w_n\}$ of unit vectors in $\mathbf{C}^n,$ we are concerned with $\D(V),$ the minimum determinant magnitude. Let $\tilde{V} = \{\tilde{w}_0,\dots,\tilde{w}_n\},$ where the $\tilde{w}_j$ are as in Lemma \ref{origininsidecomplex}. Then $\tilde{V}$ is a set of unit vectors and $\D(V)=\D(\tilde{V}).$ Hence, we can restrict ourselves to considering those sets which satisfy property $2$ of Lemma \ref{origininsidecomplex}. Furthermore, if any of the $r_j$ are equal to $0,$ then $\D(V) = 0.$ This is clearly not the largest that $\D(V)$ can be, so we should be able to assume that none of the $r_j$ are equal to $0.$ Let us say that $V = \{w_0,\dots,w_n\}$ has \textit{Property A} if there exist real $r_0,\dots,r_n,$ all strictly greater than $0,$ with $r_0 + \dots + r_n = 1$ and $r_0w_0 + \dots + r_nw_n = 0.$ We summarize these observations with the following remark. 

\begin{remark} \label{remark}
	The maximal $\D(V)$ over all sets of $n+1$ unit vectors in $\mathbf{C}^n$ is the same as the maximal $\D(V)$ over all $V$ which have Property A.
\end{remark}

\begin{proposition} \label{detasproductofcomplexdets}
	Let $V = \{w_0, \dots w_n\}$ have Property A. Let $0 \leq j_1 < j_2 < \dots < j_n \leq n$ and $0 \leq k_1 < k_2 < \dots < k_n \leq n.$ Then
	$$|\det(u_{j_1},\dots,u_{j_n},v_{k_1},\dots,v_{k_n})| = |\det(w_{j_1},\dots,w_{j_n})||\det(w_{k_1},\dots,w_{k_n})|.$$
\end{proposition}

\begin{proof}
	Assume $j_1 = 1,\dots,j_n = n$ and $k_1 = 0,\dots,k_n = n-1.$ We will show,
	$$|\det(u_1,\dots,u_n,v_0,\dots,v_{n-1}) = |\det(w_1,\dots,w_n)||\det(w_0,\dots,w_{n-1})|.$$
	The proof is the same for any other nontrivial choice. The trivial choice is when $j_1 = k_1, \dots, j_n = k_n,$ and in this case the result follows immediately from Proposition \ref{stackexchange}.
	
	Since $V$ has Property A, there exist real $r_j > 0$ such that
	$$w_0 = \sum_{j=1}^{n}\frac{-r_j}{r_0}w_j.$$
	This implies
	$$u_0 = \sum_{j=1}^{n}\frac{-r_j}{r_0}u_j.$$
	Thus,
	\begin{align*}
	|\det(u_0,\dots,u_{n-1},v_0,\dots,v_{n-1})|
	&= |\det(\sum_{j=1}^{n}(-r_j/r_0)u_j,u_1,\dots,u_{n-1},v_0,\dots,v_{n-1})|\\
	&= |\det((-r_n/r_0)u_n,u_1,\dots,u_{n-1},v_0,\dots,v_{n-1})|\\
	&= (r_n/r_0)|\det(u_1,\dots,u_n,v_0,\dots,v_{n-1})|.\end{align*}
	Again using that $V$ has Property A, we can say
	$$v_n = \sum_{j=0}^{n-1}\frac{-r_j}{r_n}v_j,$$
	and thus
	$$|\det(u_1,\dots,u_n,v_1,\dots,v_n)| = (r_0/r_n)|\det(u_1,\dots,u_n,v_0,\dots,v_{n-1})|.$$
	Therefore,
	\begin{align*}
	|\det(u_1,\dots,u_n,v_1,\dots,v_n)||\det(u_0,\dots,u_{n-1},v_0,\dots,v_{n-1})| = \\|\det(u_1,\dots,u_n,v_0,\dots,v_{n-1})|^2.
	\end{align*}
	But, by Proposition \ref{stackexchange},
	\begin{align*}
	|\det(u_1,\dots,u_n,v_1,\dots,v_n)||\det(u_0,\dots,u_{n-1},v_0,\dots,v_{n-1})| = \\ |\det(w_1,\dots,w_n)|^2|\det(w_0,\dots,w_{n-1})|^2
	\end{align*}
	and therefore
	$$|\det(u_1,\dots,u_n,v_0,\dots,v_{n-1})| = |\det(w_1,\dots,w_n)||\det(w_0,\dots,w_{n-1})|.\qedhere$$
\end{proof}

\begin{lemma} \label{nozerodets}
	If $V$ has Property A then for every $0 \leq j_1 < j_2 < \dots < j_n \leq n$ and $0 \leq k_1 < k_2 < \dots < k_n \leq n$, we have $\det({v_{j_1},\dots,v_{j_n},w_{k_1},\dots,w_{k_n}}) \neq 0.$
\end{lemma}

\begin{proof}
	This follows immediately from Proposition \ref{detasproductofcomplexdets}.
\end{proof}

\begin{lemma} \label{ptopevolumeformula}
	Let $V = \{w_0,\ldots,w_n\} \subset \mathbf{C}^n$ have property A. Then,
	$$\vol(\conv(\ddot{V})) = \frac{1}{(2n)!}\sum |\det(u_{j_1}, \ldots, u_{j_n}, v_{k_1}, \ldots, v_{k_n})|,$$
	where the sum ranges over all $0 \leq j_1 < j_2 < \dots < j_n \leq n$ and $0 \leq k_1 < k_2 < \dots < k_n \leq n.$
\end{lemma}   

\begin{proof}
	We will divide $\conv(\ddot{V})$ into $(n+1)^2$ simplices of the type 
	$$\conv(0,u_{j_1},\dots,u_{j_n},v_{k_1},\dots,v_{k_n}),$$
	which are all disjoint except at their boundaries. To this end, we must prove the following:
	\begin{enumerate}
		\item If $p \in \conv(\ddot{V})$ then it is contained in at least one of the simplices.
		\item If $p$ is a point contained in more than one of the simplices, then it is contained in a face shared by those simplices.
	\end{enumerate}
	
	First, let us prove (1). Let $p \in \conv(\ddot{V}).$ This means
	$$p = a_0u_0 + \dots + a_nu_n + b_0v_0 + \dots + b_nv_n,$$
	with $a_j,b_j \geq 0$ and 
	$$a_0 + \ldots + a_n + b_0 +\dots + b_n = 1.$$ 
	Because $V$ has Property A, there exist positive $r_j$ such that $r_0u_0 + \dots + r_nu_n = 0$ and $r_0v_0 + \dots + r_nv_n = 0$. Let $A_j=a_j/r_j$ and $B_j=b_j/r_j$. Choose $\ell$ and $k$ so that $A_{\ell} = \min\{A_j\}$ and $B_k = \min\{B_j\}$. Since
	$$0 = A_{\ell}r_0u_0 + \dots + A_{\ell}r_nu_n$$
	and
	$$0 = B_kr_0v_0 + \dots + B_kr_nv_n,$$ 
	we can subtract them both from $p$ without changing $p$. Thus we have
	$$p = \sum_{j=0}^{n} [(a_j - A_jr_j)u_j + (b_j - B_jr_j)v_j] = \sum_{j\neq \ell} [(a_j - A_{\ell}r_j)u_j] + \sum_{j\neq k} [(b_j - B_kr_j)v_j],$$
	where terms were dropped from the sum on the right hand side because $a_{\ell} - A_{\ell}r_{\ell} = b_k - B_kr_k = 0$. Also, we have that $a_j - A_{\ell}r_j \geq 0$ for each $j$ because if this is not true for some $j$, then $A_j < A_{\ell}$ which contradicts the fact that $A_{\ell}$ was the minimum. Similary, for each $j$ there is $b_j - B_kr_j \geq 0$. Since 
	$$a_0 + \dots + a_n + b_0 + \dots b_n = 1,$$ 
	it follows that 
	$$\displaystyle\sum_{j\neq\ell} (a_j - A_{\ell}r_j) + \displaystyle\sum_{j\neq k} (b_j - B_kr_j) \leq 1.$$ 
	Therefore, 
	$$p \in \conv(\{0,u_0, \dots, u_n, v_0, \dots, v_n\}\setminus \{u_{\ell},v_k\}).$$
	
	Now we prove (2). Precisely, we need to show that if 
	$$p \in \conv(\{0,u_{j_1},\dots,u_{j_n},v_{k_1},\dots,v_{k_n}\})$$ and
	$$p \in \conv(\{0,u_{\ell_1},\dots,u_{\ell_n},v_{m_1},\dots,v_{m_n}\}),$$ then
	$$p \in \conv(\{0,u_{j_1},\dots,u_{j_n},v_{k_1},\dots,v_{k_n}\}\cap \{0,u_{\ell_1},\dots,u_{\ell_n},v_{m_1},\dots,v_{m_n}\}).$$
	In particular, we need to show this for when $\{j_1,\dots,j_n\} \neq \{\ell_1,\dots,\ell_n\},$ or $\{k_1,\dots,k_n\} \neq \{m_1,\dots,m_n\},$ or both. We will consider two specific cases for the index sets to make the notation easier. For other index choices, the proof is the same. For the first case, suppose
	$$p \in \conv(\{0,u_0,\dots,u_{n-1},v_1,\dots,v_n\})$$ 
	and 
	$$p \in \conv(\{0,u_1,\dots,u_n,v_1, \dots,v_n\}).$$ 
	We want to show $p \in \conv(\{0,u_1,\dots,u_{n-1},v_1,\dots,v_n\}).$ There must exist sets of nonnegative scalars $A = \{a_j : 0 \leq j < n\},$ $B = \{b_j : 0 < j \leq n\},$ $C = \{c_j : 0 < j \leq n\},$ and $D = \{d_j : 0 < j \leq n\}$ such that 
	$$\displaystyle \sum_{j=0}^{n-1} a_j + \displaystyle \sum_{j=1}^{n} b_j \leq 1,$$ 
	$$\displaystyle \sum_{j=1}^{n} c_j + \displaystyle \sum_{j=1}^{n} d_j \leq 1,$$ 
	and
	$$p = \sum_{j=0}^{n-1} a_ju_j + \sum_{j=1}^{n} b_jv_j = \sum_{j=1}^{n} c_ju_j + \sum_{j=1}^{n} d_jv_j.$$
	Combine equations to get,
	$$0 = \sum_{j=1}^{n-1} (c_j - a_j)u_j + \sum_{j=1}^{n} (d_j - b_j)v_j + c_nu_n - a_0u_0.$$
	Since $u_0 = \displaystyle \sum_{j=1}^{n} \frac{-r_j}{r_0}u_j,$ we can say
	$$0 = \sum_{j=1}^{n-1} (c_j - a_j)u_j + \sum_{j=1}^{n} (d_j - b_j)v_j + c_nu_n - a_0\sum_{j=1}^{n} \frac{-r_j}{r_0}u_j,$$
	and then rearrange to get
	$$0 = \sum_{j=1}^{n-1} (c_j - a_j + a_0\frac{r_j}{r_0})u_j + \sum_{j=1}^{n} (d_j - b_j)v_j + (c_n + a_0\frac{r_n}{r_0})u_n.$$
	So the vectors $u_1,\dots,u_n,v_1,\dots,v_n$ have been linearly combined to get $0.$ By Lemma \ref{nozerodets}, we have $$\det(u_1,\dots,u_n,v_1,\dots,v_n) \neq 0,$$ so all the coefficients must be zero. In particular,
	$$c_n + a_0\frac{r_n}{r_0} = 0.$$
	Since $r_n,r_0 > 0$ and $c_n,a_0 \geq 0,$ it must be that $c_n = a_0 = 0.$ Therefore we can say,
	$$p = \sum_{j=1}^{n-1} a_ju_j + \sum_{j=1}^{n} b_jv_j = \sum_{j=1}^{n-1} c_ju_j + \sum_{j=1}^{n} d_jv_j,$$
	and thus $p \in \conv(\{0,u_1,\dots,u_{n-1},v_1,\dots,v_n\}).$
	
	For the second case, suppose
	$$p \in \conv(\{0,u_0,\dots,u_{n-1},v_1,\dots,v_n\})$$ 
	and 
	$$p \in \conv(\{0,u_1,\dots,u_n,v_0,\dots,v_{n-1}\}).$$
	We would like to show $p \in \conv(\{0,u_1,\dots,u_{n-1},v_1,\dots,v_{n-1}\}).$ Similar to before, we can express $p$ as
	$$p = \sum_{j=0}^{n-1} a_ju_j + \sum_{j=1}^{n} b_jv_j = \sum_{j=1}^{n} c_ju_j + \sum_{j=0}^{n-1} d_jv_j.$$
	Subtract to get
	$$0 = \sum_{j=1}^{n-1} [(c_j - a_j)u_j + (d_j - b_j)v_j] + c_nu_n + d_0v_0 - a_0u_o - b_nv_n.$$
	As before, we can use Property A to rewrite as
	$$0 = \sum_{j=1}^{n-1} [(c_j - a_j)u_j + (d_j - b_j)v_j] + c_nu_n + d_0v_0 - a_0\sum_{j=1}^{n} \frac{-r_j}{r_0}u_j - b_n\sum_{j=0}^{n-1} \frac{-r_j}{r_n}v_j,$$
	and rearrange the terms to get
	$$0 = \sum_{j=1}^{n-1} [(c_j - a_j + a_0\frac{r_j}{r_0})u_j + (d_j - b_j + b_n\frac{r_j}{r_n})v_j] + (c_n + a_0\frac{r_n}{r_0})u_n + (d_0 + b_n\frac{r_0}{r_n})v_0.$$
	Thus the vectors $u_1,\dots,u_n,v_0,\dots,v_{n-1}$ have been linearly combined to get $0.$ By Lemma \ref{nozerodets}, we have $\det(u_1,\dots,u_n,v_0,\dots,v_{n-1}) \neq 0$ and so all the coefficients must equal zero. In particular,
	$$c_n + a_0\frac{r_n}{r_0} = 0 \text{ and } d_0 + b_n\frac{r_0}{r_n} = 0.$$
	Since $r_n,r_0 > 0$ and $c_n,a_0,d_0,b_n \geq 0,$ we must have $c_n = a_0 = d_0 = b_n = 0,$ meaning we can express $p$ as
	$$p = \sum_{j=1}^{n-1} a_ju_j + \sum_{j=1}^{n-1} b_jv_j = \sum_{j=1}^{n-1} c_ju_j + \sum_{j=1}^{n-1} d_jv_j,$$
	and thus $p \in \conv(\{0,u_1,\dots,u_{n-1},v_1,\dots,v_{n-1}\}).$
\end{proof}

\begin{lemma} \label{allregdetsequalcomplex}
	Let $V = \{w_0,\dots,w_n\}$ be the vertex set of a real regular simplex. Then,
	$$|\det(u_{j_1},\dots,u_{j_n},v_{k_1},\dots,v_{k_n})| = \frac{(n+1)^{n-1}}{n^n}$$
	for all $0 \leq j_1 < j_2 < \dots < j_n \leq n$ and $0 \leq k_1 < k_2 < \dots < k_n \leq n$.
\end{lemma}

\begin{proof}
	By Proposition \ref{detasproductofcomplexdets}, $$|\det(u_{j_1},\dots,u_{j_n},v_{k_1},\dots,v_{k_n})| = |\det(w_{j_1},\dots,w_{j_n})||\det(w_{k_1},\dots,w_{k_n})|.$$ 
	By Proposition \ref{regdetsformula}, we have $$|\det(w_{j_1},\dots,w_{j_n})||\det(w_{k_1},\dots,w_{k_n})| = \frac{(n+1)^{n-1}}{n^n}.$$
	Therefore,
	$$|\det(u_{j_1},\dots,u_{j_n},v_{k_1},\dots,v_{k_n})| = \frac{(n+1)^{n-1}}{n^n}.\qedhere$$
\end{proof}

\begin{theorem} \label{main result}
	Let $V = \{w_0,\dots,w_n\} \subset \mathbf{C}^n$ be a set of unit vectors. Then $D(V)$ is maximized when $w_0,\dots,w_n$ are the vertices of a real regular simplex. In that case,
	$$D(V) = \sqrt{\frac{(n+1)^{n-1}}{n^n}}.$$
\end{theorem}

\begin{proof}
	Let $R \subset \mathbf{C}^n$ be a set of unit vectors and the vertex set of a real regular simplex in $\mathbf{C}^n.$ Let $V = \{w_0,\dots,w_n\}$ be a set of unit vectors with Property A. Then by Theorem \ref{maxvolumecomplex}, we have $\vol(\conv(\ddot V)) \leq \vol(\conv(\ddot R))$. By Lemma~\ref{ptopevolumeformula},
	$$\vol(\conv(\ddot V)) = \frac{1}{(2n)!}\sum |\det(u_{j_1}.\dots,u_{j_n},v_{k_1},\dots,v_{k_n})|.$$
	It can be easily seen that $R$ has Property A. Hence, by Lemma \ref{ptopevolumeformula} and Lemma \ref{allregdetsequalcomplex}, we have
	$$\vol(\conv(\ddot R)) = \frac{(n+1)^2}{(2n)!}\frac{(n+1)^{n-1}}{n^n}.$$
	So we have
	$$\sum |\det(u_{j_1}.\dots,u_{j_n},v_{k_1},\dots,v_{k_n})| \leq (n+1)^2\frac{(n+1)^{n-1}}{n^n}.$$
	Let 
	$$d = \min \{|\det(u_{j_1}.\dots,u_{j_n},v_{k_1},\dots,v_{k_n})|\}.$$ 
	Then we have, 
	$$d \leq \frac{(n+1)^{n-1}}{n^n}.$$ 
	By Proposition \ref{detasproductofcomplexdets},
	$$|\det(u_{j_1}.\dots,u_{j_n},v_{k_1},\dots,v_{k_n})| = |\det(w_{j_1},\dots,w_{j_n})||\det(w_{k_1},\dots,w_{k_n})|.$$
	The right hand side is minimized when $\{j_1,\dots,j_n\} = \{k_1,\dots,k_n\} = \{\ell_1,\dots,\ell_n\}$, where $\{\ell_1,\dots,\ell_n\}$ is such that $D(V) = |\det(w_{\ell_1},\dots,w_{\ell_n})|$. Therefore,
	$$D(V)^2 = d \leq \frac{(n+1)^{n-1}}{n^n},$$
	and so,
	$$D(V) \leq \sqrt{\frac{(n+1)^{n-1}}{n^n}}.$$
	By Remark \ref{remark}, this is the maximum over all $V$ and not just those $V$ with Property A.
\end{proof}

\begin{corollary}
	Let $V = \{w_0,\dots,w_n\}$ be a set of $n+1$ unit vectors in $\mathbf{R}^n.$ Then $\D(V)$ is maximized when $w_0,\dots,w_n$ correspond to the vertices of a regular simplex. In that case,
	$$\D(V) = \sqrt{\frac{(n+1)^{n-1}}{n^n}}.$$
\end{corollary}

\begin{proof}
	If there were some set $V$ which has a greater $\D(V),$ then this would contradict Theorem \ref{main result}.
\end{proof}

\section{Some Motivation and Further Cases} \label{extras}

\subsection{Motivation from Complex Function Theory}

Recall Picard's Theorem.

\begin{theorem*}[{Picard's Theorem \cite[\S8.3]{Ahlfors}}]\label{picardthm}
	If a meromorphic function never takes on any of the three values $0,$ $1,$
	and $\infty,$ then the meromorphic function must be constant.
\end{theorem*}

To make a connection with the maximin determinants problem, Picard's Theorem can be
reformulated as follows. Let $F(z)=(f_0(z),f_1(z))$ be a vector valued
function with coordinate functions $f_0$ and $f_1.$ This vector valued
function is an alternate representation for the meromorphic function
$f_1/f_0.$ Consider the three unit vectors $$
v_0=(1,0),\qquad v_1=(0,1),\qquad\textnormal{and}\qquad
v_2=\left(\frac{1}{\sqrt{2}},-\frac{1}{\sqrt{2}}\right)
$$
in $\mathbf{C}^2.$ Then, if the meromorphic function $f_1/f_0$ omits
the values $0,$ $1,$ and $\infty,$ then the following three dot products
$$
v_0\cdot F = f_0, \qquad
v_1\cdot F = f_1, \qquad\textnormal{and}\qquad
v_2\cdot F = f_0-f_1
$$
never vanish. Thus, Picard's Theorem can be reformulated by saying that
if none of the three dot products vanish, then $f_1/f_0$ must be constant.
This statement was generalized by Bloch and Cartan to higher dimensions.
\begin{theorem*}[{\cite[Th~3.10.6]{Kobayashi}}]Let $F(z)=(f_0(z),\dots,f_n(z))$
	be a $\mathbf{C}^{n+1}$ valued function of a complex variable $z,$
	and assume that $F(z)$ is never the zero vector.
	Let $v_0,\dots,v_{2n}$ be $2n+1$ unit vectors in $\mathbf{C}^{n+1}$ such
	that any $n+1$ of them are linearly independent. If the $2n+1$
	dot products $v_0\cdot F,$ \dots, $v_{2n}\cdot F$ are all never zero,
	then there is a single complex valued function $f(z)$ and complex 
	constants $c_0,\dots,c_n$ such that $f_j=c_jf$ for $j=0,\dots,n.$
\end{theorem*}
\par
Picard's Theorem and its generalization by Bloch and Cartan are theorems about
entire functions. A principle formulated by A.~Bloch
(see \cite[Ch.~VIII]{Lang}) says that to each such theorem about entire 
functions, there should be a corresponding theorem for functions analytic
in the unit disc. The following theorem of Landau is the analog of the
Picard Theorem.
\begin{theorem*}[{Landau's Theorem \cite{Landau}}]If $f$ is analytic
	in the unit disc and
	$f$ never takes on the values $0$ or $1,$ then $|f'(0)|$ can be explicitly
	bounded in terms of $|f(0)|.$
\end{theorem*}
Cherry and Eremenko \cite{ChEr} gave a Landau-type counterpart to the
higher dimensional result of Bloch and Cartan. In that work, they gave
an explicit derivative estimate consisting of two factors. One factor is
a constant depending only on the dimension $n.$ The other factor is a 
geometric factor depending on the configuration of the unit vectors
$v_0,\dots,v_{2n}.$ The connection to this paper is that the geometric
factor is closely related to the minimum absolute value of the determinants
of the vectors taken $n+1$ at a time. Cherry and Eremenko's bound,
although explicit, is almost
certainly far from optimal. In particular, it is not clear if the factor
depending on dimension alone is necessary. As an initial foray into 
investigating this, Cherry asked if one could find configurations of
$2n+1$ unit vectors in $\mathbf{C}^{n+1}$ in such a way so that as
the dimension $n$ tends to infinity, the minimum absolute value of the
various determinants stays bounded away from zero. Theorem~\ref{main result} shows that no such configuration is possible, because
$$\lim_{n \to \infty} \sqrt{\frac{(n+1)^{n-1}}{n^n}} = 0.$$
This means that to investigate whether or not the dimension only factor is
needed, one needs to find examples of vector valued functions 
in higher and higher dimensions whose
derivatives grow faster than the geometric factor
alone in the Cherry and Eremenko theorem allows.

\subsection{$k+1$ Vectors in $\mathbf{R}^2$} \label{real2d}

In this section we see the solution to the maximin determinants problem for $k+1$ vectors in $\mathbf{R}^2,$ for arbitrary $k+1>n.$ Define $\theta_{i,j}$ as the angle from $v_i$ to $v_j$ in radians, measured in the counterclockwise direction. Let $v_0,\dots,v_k \in \mathbf{R}^2$ be vertices of a convex $p$-gon inscribed in $S^1,$ with $p>k.$ Say $v_0,\dots,v_k$ are \textit{consecutive vertices} if $v_{j+1}$ is the vertex of the $p$-gon adjacent to $v_j$ in the counterclockwise direction for all $0 \leq j < k.$

\begin{theorem}
	Let $V = \{v_0,\dots,v_k\} \subset \mathbf{R}^2$ be a set of unit vectors. Then $D(V)$ is maximal when $v_0,\dots,v_k$ are consecutive vertices of a regular $(2k+2)$-gon inscribed in $S^1,$ and for that maximal configuration 
	$$D(V) = \sin[\pi/(k+1)].$$ 
\end{theorem}

\begin{proof}
	Let $V = \{v_0,\dots,v_k\} \subset \mathbf{R}^2$ be unit vectors. Observe that $|\det(v_i,v_j)| = |\sin\theta_{i,j}|$ for all $0 \leq i < j \leq k.$ So we will maximize 
	$$D(V) = \min\{|\sin(\theta_{i,j})|:0 \leq i < j \leq k\}.$$
	
	Without loss of generality, assume $v_0$ lies on the $x$ axis. Assume all the $v_j$ lie in the first or second quadrant (if some $v_j$ does not, then multiply it by $-1$ which will not change any of the determinant magnitudes). If they are not already, relabel $v_0,\dots,v_k$ so they are consecutive vertices of $\conv(V).$ In summary, if $k=5$ then we assume a configuration like in figure \ref{figure3}.
	
	\begin{figure}
		\centering
		\includegraphics[width=0.6\textwidth]{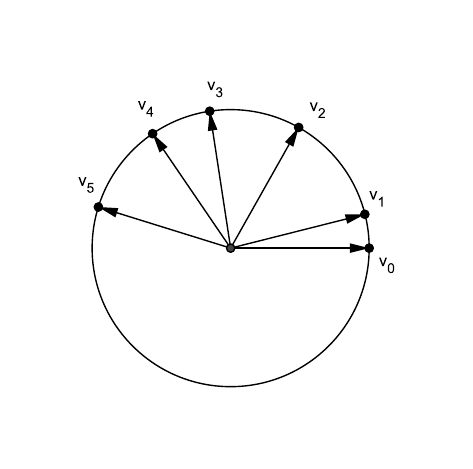}
		\caption{}\label{figure3}	
	\end{figure}
	
	We have 
	$$\theta_{0,1} + \theta_{1,2} + \dots + \theta_{k-1,k} + (\pi - \theta_{0,k}) = \pi.$$
	This implies the minimum of $\theta_{0,1},\dots,\theta_{k-1,k}$ and $\pi - \theta_{0,k}$ is less than or equal to $\pi/(k+1),$ which gives the minimum of $\sin(\theta_{0,1}),\dots,\sin(\theta_{k-1,k}),$ and $\sin(\theta_{0,k})$ is less than or equal to $\sin[\pi/(k+1)].$ This means that for any configuration $V,$ we have $D(V) \leq \sin[\pi/(k+1)].$
	
	If $V = \{v_0,\dots,v_k\}$ is such that 
	$$\theta_{0,1} = \dots = \theta_{k-1,k} = \pi - \theta_{0,k} = \pi/(k+1),$$ 
	then 
	$$\sin(\theta_{0,1}) = \dots = \sin(\theta_{k-1,k}) = \sin(\theta_{0,k}) = \sin[\pi/(k+1)].$$ 
	Further, $\pi/(k+1) \leq \theta_{i,j} \leq \pi - \pi/(k+1)$ for all $0 \leq i < j \leq k$ which implies $\sin(\theta_{i,j}) \geq \sin[\pi/(k+1)]$ for all $0 \leq i < j \leq k.$ Thus, for this configuration $V,$ we have $D(V) = \sin[\pi/(k+1)].$ Since we have shown that for any $V,$ there is $D(V) \leq  \sin[\pi/(k+1)],$ we have maximized $D(V).$ Since $\theta_{j,j+1} = \pi/(k+1)$ for all $0 \leq j < k,$ the vectors of $V$ are $k$ consecutive vertices of the regular $(2k +2)$-gon inscribed in $S^1.$
\end{proof}

Note that when $k$ is even and $v_0,\dots,v_k$ are the vertices of the regular $(k+1)$-gon, flipping the $v_j$ so they all lie in the first or second quadrant produces $k+1$ consecutive vertices of the regular $(2k+2)$-gon. This is as expected, since the solution to the problem of $n+1$ unit vectors in $\mathbf{R}^n$ said the equaliteral triangle maximizes $D(V)$ for $n=2$ and $k+1=3.$

\subsection{The Spherical Code Problem and $k+1$ Vectors in $\mathbf{C}^2$}

How should one place $k+1$ points on the surface of $S^2$ so the minimum distance among all pairs of points is maximized, and what is this distance? This is a classical problem in geometry, which we call the spherical code problem. The solution is known for some small values of $k+1,$ but is open in general. For more information, see \cite{Weisstein} or \cite{Sloane}.   

There is a connection between the maximin determinants problem for $k+1$ vectors in $\mathbf{C}^2$ and the spherical code problem for $k+1$ points on $S^2$. It can be seen that the absolute value of the determinant of two unit vectors in $\mathbf{C}^2$ is equal to the distance between two points on a sphere in $\mathbf{R}^3$ of radius $1/2$, where those points are obtained by stereographic projection from representatives from the two complex vectors when viewed as points on the complex projective line. For the details of this, see \cite[Pg 14]{ChYe} and its errata. As a result, a solution to one problem entails a solution to the other. Since the spherical code problem is unsolved and considered hard for most values of $k+1,$ the maximin determinants problem in $\mathbf{C}^2$ is probably also hard for most values of $k+1.$

\subsection{More than $n+1$ Vectors in $\mathbf{R}^n$}

To solve the problem for $n+1$ unit vectors in $\mathbf{C}^n,$ we maximized the minimum of a certain set of determinants coming from $2n+2$ unit vectors in $\mathbf{R}^{2n}.$ This did not solve the maximin determinants problem for $2n+2$ vectors in $\mathbf{R}^{2n}$ because we did not consider the minimum over all possible determinants of the $2n+2$ vectors, and because we only considered special configurations of vectors coming from the $n+1$ vectors in $\mathbf{C}^n.$ So, there is still work to be done for the case of $n+2$ vectors in $\mathbf{R}^n.$

One strategy for values of $k>n$ is to attempt to generalize the methods used for $k+1$ vectors in $\mathbf{R}^2.$ As seen in subsection \ref{real2d}, the optimal configuration of $k+1$ vectors in $\mathbf{R}^2$ is $k+1$ consecutive vertices of a regular $(2k+2)$-gon. We might then hope that in three dimensions the optimal $k+1$ vectors would come from some special polyhedron with $2k+2$ vertices.

As a matter of fact, if $v_0,v_1,v_2$ and $v_3$ are the vertices of a regular tetrahedron, then $v_0,v_1,v_2,v_3,-v_0,-v_1,-v_2,$ and $-v_3$ are the vertices of a cube. This leads to the following conjecture.

\begin{conjecture*}
	For $6$ unit vectors in $\mathbf{R}^3,$ an optimal configuration is the $6$ vertices of an icosahedron contained in the northern hemisphere if one of the vertices lies at the north pole.
\end{conjecture*}

Similarly, we may ask the following.

\begin{question*}
	For $10$ unit vectors in $\mathbf{R}^3,$ is an optimal configuration the $10$ vertices in the northern hemisphere of a dodecahedron if one of the vertices lies at the north pole? 
\end{question*}

We might also ask if the optimal $k+1$ vectors could be vertices of a polyhedron with $2k+2$ vertices which either has maximal volume over all polyhedra with $2k+2$ vertices inscribed in the unit sphere, or which solves the spherical code problem for $2k+2$ points. Considering $k=3$ again, the $8$ vectors which maximize volume and the $8$ vectors which solve the spherical code problem are known, and can be found, for instance, in \cite{Sloane}. The configurations which answer each problem are not the same. In either case, however, one can compute the determinants and see that it is not possible to choose $4$ of the $8$ vectors so that the absolute values of the determinants are equal to those of the regular tetrahedron.



\end{document}